\documentclass[reqno,12pt]{amsart}
\usepackage{amscd,amsfonts,amssymb}
\usepackage[T2A]{fontenc}
\usepackage[cp1251]{inputenc}
\numberwithin{equation}{section}
\newtheorem{Theorem}{Theorem}[section]
\newtheorem{Lemma}{Lemma}[section]

\theoremstyle{definition}

\theoremstyle{remark}
\newtheorem{Remark}{Remark}[section]
\newtheorem{Example}{Example}[section]

\newcommand{\mes}{\mathop{\rm mes}\nolimits}

\author{A.A. Kon'kov}
\address{Department of Differential Equations,
Faculty of Mechanics and Mathematics,
Mo\-s\-cow Lo\-mo\-no\-sov State University,
Vorobyovy Gory,
Moscow, 119992 Russia.
}
\email{konkov@mech.math.msu.su}
\author{A.E. Shishkov}
\address{
Center of Nonlinear Problems of Mathematical Physics,
RUDN University,
Miklukho-Maklaya str. 6,
Moscow, 117198 Russia.
}
\email{aeshkv@yahoo.com}
\thanks{
}
\title[On a Dini-type blow-up condition]{On a Dini type blow-up condition for nonlinear higher order differential inequalities}
\date{}

\begin{document}

\begin{abstract}
We obtain a Dini type blow-up condition for global weak solutions of the differential inequality
$$
	\sum_{|\alpha| = m}
	\partial^\alpha
	a_\alpha (x, u)
	\ge
	g (|u|)
	\quad
	\mbox{in } {\mathbb R}^n,
$$
where $m, n \ge 1$ are integers and $a_\alpha$ and $g$ are some functions.
\end{abstract}

\maketitle

\section{Introduction}

We consider the differential inequality
\begin{equation}
	\sum_{|\alpha| = m}
	\partial^\alpha
	a_\alpha (x, u)
	\ge
	g (|u|)
	\quad
	\mbox{in } {\mathbb R}^n,
	\label{1.1}
\end{equation}
where $m, n \ge 1$ are integers and $a_\alpha$ are Caratheodory functions such that
$$
	|a_\alpha (x, \zeta)| 
	\le 
	A |\zeta|,
	\quad
	|\alpha| = m,
$$
with some constant $A > 0$ for almost all $x = {(x_1, \ldots, x_n)} \in {\mathbb R}^n$ 
and for all $\zeta \in {\mathbb R}$.
As is customary, by $\alpha = {(\alpha_1, \ldots, \alpha_n)}$ we mean a multi-index with
$|\alpha| = \alpha_1 + \ldots + \alpha_n$ 
and
$
	\partial^\alpha 
	= 
	{\partial^{|\alpha|} / (\partial_{x_1}^{\alpha_1} \ldots \partial_{x_n}^{\alpha_n})}.
$
In so doing, it is assumed that $g$ is a non-decreasing convex function on the interval $[0, \infty)$ such that $g (\zeta) > 0$ for all $\zeta > 0$.

We denote by $B_r^x$ the open ball in ${\mathbb R}^n$ of radius $r > 0$ and center at $x$. For $x = 0$, let us write $B_r$ instead of $B_r^0$.

A function $u  \in {L_{1, loc} ({\mathbb R}^n)}$ is called a global weak solution of~\eqref{1.1} 
if ${g (|u|)} \in {L_{1, loc} ({\mathbb R}^n)}$ and
\begin{equation}
	\int_{{\mathbb R}^n}
	\sum_{|\alpha| = m}
	(-1)^m
	a_\alpha (x, u)
	\partial^\alpha
	\varphi
	\,
	dx
	\ge
	\int_{{\mathbb R}^n}
	g (|u|)
	\varphi
	\,
	dx
	\label{1.2}
\end{equation}
for any non-negative function $\varphi \in C_0^\infty ({\mathbb R}^n)$.

The absence of nontrivial solutions of differential equations and inequalities or, in other words, the blow-up phenomenon has traditionally attracted the interest of many mathematicians~[1--16].
In so doing, most authors dealt with second-order differentila operators or limited themselves to the case of power-law nonlinearity $g (t) = t^\lambda$. The case of general nonlinearity for higher order differential operators was considered in paper~\cite{Nonlinearity}. 
In the present paper, we managed to obtain a Dini-type blow-up condition that enhances the results of~\cite{Nonlinearity}. For power-law nonlinearity this enhances blow-up conditions given in~\cite{GMP, NiSerrin} and, in particular, the well-known W.-M.~Ni and J.~Serrin condition.

\section{Main results}

\begin{Theorem}\label{T2.1}
Let $n > m$ and, moreover,
\begin{equation}
	\int_1^\infty
	g^{- 1 / m} (\zeta)
	\zeta^{1 / m - 1}
	\,
	d\zeta
	<
	\infty
	\label{T2.1.1}
\end{equation}
and
\begin{equation}
	\int_0^1
	\frac{
		g (r)
		\,
		dr
	}{
		r^{1 + n / (n - m)}
	}
	=
	\infty.
	\label{T2.1.2}
\end{equation}
Then any global weak solution of~\eqref{1.1} is trivial.
\end{Theorem}

\begin{Remark}\label{R2.1}
Earlier in~\cite[Theorem~2.5]{Nonlinearity}, it was proved that, in the case where $n \le m$ and condition~\eqref{T2.1.1} is valid, any global weak solution of~\eqref{1.1} is trivial.
\end{Remark}

Proof of Theorem~\ref{T2.1} is given in Section~\ref{proof}. Now, we demonstrate its application.

\begin{Example}\label{E2.1}
Consider the inequality 
\begin{equation}
	\sum_{|\alpha| = m}
	\partial^\alpha
	a_\alpha (x, u)
	\ge
	c_0 
	|u|^\lambda
	\quad
	\mbox{in } {\mathbb R}^n,
	\quad
	c_0 = const > 0,
	\label{E2.1.1}
\end{equation}
where $n > m$ and $\lambda$ is a real number. By Theorem~\ref{T2.1}, if
\begin{equation}
	1 < \lambda \le \frac{n}{n - m}
	\label{E2.1.2}
\end{equation}
then any global weak solution of~\eqref{E2.1.1} is trivial.
It is well-known that condition~\eqref{E2.1.2} can not be improved in the class of power-law nonlinearities~\cite{KS, MPbook}.
In the case of $m = 2$, formula~\eqref{E2.1.2} coincides with W.-M.~Ni and J.~Serrin condition~\cite{NiSerrin}.
\end{Example}

\begin{Example}\label{E2.2}
We examine the critical exponent $\lambda = n / (n - m)$ in~\eqref{E2.1.2}.
Namely, consider the inequality
\begin{equation}
	\sum_{|\alpha| = m}
	\partial^\alpha
	a_\alpha (x, u)
	\ge
	c_0 
	|u|^{n / (n - m)}
	\log^\mu 
	\left(
		e + \frac{1}{|u|}
	\right)
	\quad
	\mbox{in } {\mathbb R}^n,
	\quad
	c_0 = const > 0,
	\label{E2.2.1}
\end{equation}
where $n > m$ and $\mu$ is a real number.
In so doing, in the case of $u = 0$, we extend by continuity the right-hand side of~\eqref{E2.2.1} with zero.

According to Theorem~\ref{T2.1}, if
$$
	\mu \ge - 1,
$$
then any global weak solution of~\eqref{E2.2.1} is trivial.
\end{Example}

\section{Proof of Theorem~\ref{T2.1}}\label{proof}

In this section, by $C$ and $\sigma$ we denote various positive constants that can depend only on $A$, $m$, and $n$. 
We need the following two known results.

\begin{Theorem}\label{T3.1}
Let~\eqref{T2.1.1} be valid, then
$$
	\lim_{r \to \infty}
	\frac{1}{r^n}
	\int_{B_r}
	|u|
	\,
	dx
	=
	0
$$
for any global weak solution of inequality~\eqref{1.1}.
\end{Theorem}

\begin{Lemma}\label{L3.1}
Let $u$ be a global weak solution of~\eqref{1.1}, then
$$
	\int_{
		B_{r_2}
		\setminus
		B_{r_1}
	}
	|u|
	\,
	dx
	\ge
	C
	(r_2 - r_1)^m
	\int_{
		B_{r_1}
	}
	g (|u|)
	\,
	dx
$$
for all real numbers $0 < r_1 < r_2$ such that $r_2 \le 2 r_1$. 
\end{Lemma}

Proof of Theorem~\ref{T3.1} and Lemma~\ref{L3.1} is given in~\cite[Theorem~2.4 and Lemma~3.1]{Nonlinearity}.

\medskip

From now on, we denote
$$
	E (r)
	=
	\int_{
		B_r
	}
	g (|u|)
	\,
	dx,
	\quad
	r > 0.
$$

\begin{Lemma}\label{L3.2}
Let $u$ be a global weak solution of~\eqref{1.1}, then
\begin{equation}
	E (r) - E (r / 2)
	\ge
	C
	r^n
	g 
	\left(
		\frac{
			\sigma
		}{
			r^{n - m}
		}
			E (r / 2)
	\right)
	\label{L3.2.1}
\end{equation}
for all real numbers $r > 0$.
\end{Lemma}

\begin{proof}
From Lemma~\ref{L3.1} with $r_1 = r / 2$ and $r_2 = r$, it follows that
$$
	\frac{
		1
	}{
		\mes B_r \setminus B_{r / 2}
	}
	\int_{
		B_r
		\setminus
		B_{r / 2}
	}
	|u|
	\,
	dx
	\ge
	\frac{
		\sigma
	}{
		r^{n - m}
	}
	\int_{
		B_{r / 2}
	}
	g (|u|)
	\,
	dx.
$$
Since $g$ is a non-decreasing function, this yields
$$
	g
	\left(
		\frac{
			1
		}{
			\mes B_r \setminus B_{r / 2}
		}
		\int_{
			B_r
			\setminus
			B_{r / 2}
		}
		|u|
		\,
		dx
	\right)
	\ge
	g
	\left(
		\frac{
			\sigma
		}{
			r^{n - m}
		}
		\int_{
			B_{r / 2}
		}
		g (|u|)
		\,
		dx
	\right).
$$
Due to the convexity of the function $g$, we also have
$$
	\frac{
		1
	}{
		\mes B_r \setminus B_{r / 2}
	}
	\int_{
		B_r \setminus B_{r / 2}
	}
	g (|u|)
	\,
	dx
	\ge
	g
	\left(
		\frac{
			1
		}{
			\mes B_r \setminus B_{r / 2}
		}
		\int_{
			B_r \setminus B_{r / 2}
		}
		|u|
		\,
		dx
	\right).
$$
Thus, combining the last two inequalities, one can conclude that
$$
	\frac{
		1
	}{
		\mes B_r \setminus B_{r / 2}
	}
	\int_{
		B_r \setminus B_{r / 2}
	}
	g (|u|)
	\,
	dx
	\ge
	g
	\left(
		\frac{
			\sigma
		}{
			r^{n - m}
		}
		\int_{
			B_{r / 2}
		}
		g (|u|)
		\,
		dx
	\right)
$$
for all real numbers $r > 0$. This immediately implies~\eqref{L3.2.1}.
\end{proof}

\begin{proof}[Proof of Theorem~$\ref{T2.1}$] 
Assume the converse. Let $u$ be a nontrival global weak solution of~\eqref{1.1}. 
By Lemma~\ref{L3.1}, 
$$
	\frac{
		E (r)
	}{
		r^{n-m}
	}
	\le
	\frac{C}{r^n}
	\int_{
		B_{2 r}
		\setminus
		B_r
	}
	|u|
	\,
	dx
$$
for all real numbers $r > 0$, whence in accordance with Theorem~\ref{T3.1} it follows that
$$
	\lim_{r \to \infty}
	\frac{
		E (r)
	}{
		r^{n-m}
	}
	=
	0.
$$
Take a real number $r_0 > 0$ such that $E (r_0) > 0$. We put $r_i = 2^i r_0$, $i = 1,2,\ldots$.
Obviously, there are sequences of integers $0 < s_i < l_i \le s_{i+1}$, $i = 1,2,\ldots$, such that
$$
	\frac{
		E (r_j)
	}{
		r_j^{n - m}
	}
	>
	\frac{
		E (r_{j+1})
	}{
		r_{j+1}^{n - m}
	}
$$
for all $j \in \Xi$ and, moreover,
$$
	\frac{
		E (r_j)
	}{
		r_j^{n - m}
	}
	\le
	\frac{
		E (r_{j+1})
	}{
		r_{j+1}^{n - m}
	}
$$
for all $j \not\in \Xi$, where 
$$
	\Xi = \bigcup_{i=1}^\infty [s_i, l_i).
$$
Since $E$ is a non-decreasing function,
we obtain
\begin{equation}
	\frac{
		2^{n - m} 
		E (r_{j+1})
	}{
		r_{j+1}^{n - m}
	}
	\ge
	\frac{
		E (r_j)
	}{
		r_j^{n - m}
	}
	>
	\frac{
		E (r_{j+1})
	}{
		r_{j+1}^{n - m}
	}
	\label{PT2.1.1}
\end{equation}
for all $j \in \Xi$.
By Lemma~\ref{L3.1},
$$
	E (r_{j+1}) - E (r_j)
	\ge
	C
	r_j^n
	g 
	(
		\sigma
		r_j^{- n + m} 
		E (r_j)
	),
$$
whence it follows that
$$
	\frac{
		E (r_{j+1}) - E (r_j)
	}{
		E^{
			n / (n - m)
		}
		(r_j)
	}
	\ge
	C
	h 
	(
		\sigma 
		r_j^{
			- n + m
		} 
		E (r_j)
	)
$$
for all $j = 1,2,\ldots$, where
$$
	h (\zeta)
	=
	\frac{
		g (\zeta)
	}{
		\zeta^{
			n / (n - m)
		}
	}.
$$
Multiplying this by the inequality
$$
	1 
	\ge
	\frac{
		r_j^{-n + m} E (r_j) - r_{j+1}^{-n + m} E (r_{j+1})
	}{	
		r_j^{-n + m} E (r_j)
	},
$$
we have
\begin{equation}
	\frac{
		E (r_{j+1}) - E (r_j)
	}{
		E^{
			n / (n - m)
		}
		(r_j)
	}
	\ge
	C
	\frac{
		h 
		(
			\sigma 
			r_j^{
				- n + m
			} 
			E (r_j)
		)
	}{
		r_j^{-n + m} E (r_j)
	}
	\left(
		r_j^{-n + m} E (r_j) - r_{j+1}^{-n + m} E (r_{j+1})
	\right)
	\label{PT2.1.2}
\end{equation}
for all $j \in \Xi$. 
In view of~\eqref{PT2.1.1} and the monotonicity of the function $E$, one can assert that
$$
	E (r_{j+1})
	\ge
	E (r_j)
	>
	\frac{
		1
	}{
		2^{n - m}
	}
	E (r_{j+1})
$$
for all $j \in \Xi$.
Consequently, we have
$$
	\int_{
		E (r_j)
	}^{
		E (r_{j+1})
	}
	\frac{
		d\zeta
	}{
		\zeta^{
			n / (n - m)
		}
	}
	\ge
	C
	\frac{
		E (r_{j+1}) - E (r_j)
	}{
		E^{
			n / (n - m)
		}
		(r_j)
	}
$$
for all $j \in \Xi$.
It can be also seen that~\eqref{PT2.1.1} implies that
$$
	\frac{
		h 
		(
			\sigma 
			r_j^{
				- n + m
			} 
			E (r_j)
		)
	}{
		r_j^{-n + m} E (r_j)
	}
	\left(
		r_j^{-n + m} E (r_j) - r_{j+1}^{-n + m} E (r_{j+1})
	\right)
	\ge
	C
	\int_{
		r_{j+1}^{-n + m} E (r_{j+1})
	}^{
		r_j^{-n + m} E (r_j)
	}
	\frac{
		\tilde h (\sigma \zeta)
	}{
		\zeta
	}
	d \zeta
$$
for all $j \in \Xi$, where
$$
	\tilde h (\zeta)
	=
	\inf_{
		(\zeta, \, 2^{n - m} \zeta)
	}
	h.
$$
Thus, taking into account~\eqref{PT2.1.2}, we obtain
$$
	\int_{
		E (r_j)
	}^{
		E (r_{j+1})
	}
	\frac{
		d\zeta
	}{
		\zeta^{
			n / (n - m)
		}
	}
	\ge
	C
	\int_{
		r_{j+1}^{-n + m} E (r_{j+1})
	}^{
		r_j^{-n + m} E (r_j)
	}
	\frac{
		\tilde h (\sigma \zeta)
	}{
		\zeta
	}
	d \zeta
$$
for all $j \in \Xi$.
In its turn, summing the last expression over all $j \in \Xi$, we arrive at the relation
$$
	\sum_{i=1}^\infty
	\int_{
		E (r_{s_i})
	}^{
		E (r_{l_i})
	}
	\frac{
		d\zeta
	}{
		\zeta^{
			n / (n - m)
		}
	}
	\ge
	C
	\sum_{i=1}^\infty
	\int_{
		r_{l_i}^{-n + m} E (r_{l_i})
	}^{
		r_{s_i}^{-n + m} E (r_{s_i})
	}
	\frac{
		\tilde h (\sigma \zeta)
	}{
		\zeta
	}
	d \zeta.
$$
Since 
$
	E (r_{s_{i+1}}) \ge E (r_{l_i}) 
$
and
$
	r_{l_i}^{-n + m}  
	E (r_{l_i})
	\le
	r_{s_{i+1}}^{-n + m}  
	E (r_{s_{i+1}})
$
for all integers $i > 1$ and, moreover, 
$$
	\lim_{i \to \infty}
	r_{l_i}^{-n + m}  
	E (r_{l_i})
	=
	0,
$$
this implies the estimate
\begin{equation}
	\int_{
		E (r_{s_1})
	}^\infty
	\frac{
		d\zeta
	}{
		\zeta^{
			n / (n - m)
		}
	}
	\ge
	C
	\int_0^{
		r_{s_1}^{-n + m} E (r_{s_1})
	}
	\frac{
		\tilde h (\sigma \zeta)
	}{
		\zeta
	}
	d \zeta.
	\label{PT2.1.3}
\end{equation}
It is obvious that $E (r_{s_1}) \ge  E (r_0) > 0$; therefore,
$$
	\int_{
		E (r_{s_1})
	}^\infty
	\frac{
		d\zeta
	}{
		\zeta^{
			n / (n - m)
		}
	}
	=
	\frac{m}{n - m}
	E^{- m / (n - m)} (r_{s_1})
	<
	\infty.
$$
At the same time, from the monotonicity of the function $g$, it follows that
$$
	\tilde h (\zeta)
	\ge
	\frac{1}{2^n}
	\frac{
		g (\zeta)
	}{
		\zeta^{
			n / (n - m)
		}
	}.
$$
Thus,~\eqref{PT2.1.3} leads to the inequality
$$
	\int_0^{
		r_{s_1}^{-n + m} E (r_{s_1})
	}
	\frac{
		g (\sigma \zeta)
		\,
		d \zeta
	}{
		\zeta^{
			1 + n / (n - m)
		}
	}
	<
	\infty
$$
which contradicts~\eqref{T2.1.2}.
The proof is completed.
\end{proof}

\end{document}